\documentclass[reqno,11pt]{amsart}
\usepackage{amsmath, amsthm, amssymb, geometry, latexsym, mathrsfs}
\usepackage{bm}
\usepackage[all]{xy}
\xyoption{poly}
\usepackage{color, hyperref}
\usepackage{cite}
\usepackage[svgnames,psnames]{xcolor}
\hypersetup{
    colorlinks=true,
    citecolor=DarkBlue,
    linkcolor=FireBrick,
    urlcolor=orange,
}

\def\al{\alpha}
\def\be{\beta}

\def\Ga{\Gamma}

\def\De{\Delta}
\def\ep{\varepsilon}
\def\e{\varepsilon}

\def\La{\Lambda}

\def\si{\sigma}

\def\ang#1{{\langle #1 \rangle}}

\def\Z{{\mathbb Z}}
\def\N{{\mathbb N}}
\def\F{{\mathbb F}}

\def\p{{\mathfrak p}}

\def\dim{\operatorname{dim}}

\def\Ext{\operatorname{Ext}}

\def\gldim{\operatorname{gldim}}

\def\Hom{\operatorname{Hom}}

\def\injdim{\operatorname{injdim}}

\def\mnull{{\operatorname{null}}}

\def\op{\operatorname{op}}

\def\rank{\operatorname{rank}}

\def\Spec{\operatorname{Spec}}

\def\grmod{\operatorname{\mathsf{grmod}}}
\def\GrMod{\operatorname{\mathsf{GrMod}}}

\def\CM{\operatorname{\mathsf{CM}}^{\Z}}

\def\qgr{\operatorname{\mathsf{qgr}}}

\def\uCM{\operatorname{\underline{\mathsf{CM}}}^{\Z}}

\def\fdim{\operatorname{\mathsf{fdim}}}

\def\Db{\mathsf{D^b}}
\def\Kb{\mathsf{K^b}}
\def\mod{\operatorname{\mathsf{mod}}}

\def\proj{\operatorname{\mathsf{proj}}}

\def\<{\langle}
\def\>{\rangle}
\def\rnum#1{\expandafter{\romannumeral #1}}
\def\Rnum#1{\uppercase\expandafter{\romannumeral #1}}

\theoremstyle{plain} 
\newtheorem{thm}{Theorem}[section]
\newtheorem{cor}[thm]{Corollary}
\newtheorem*{thm*}{Theorem}
\newtheorem*{cor*}{Corollary}
\newtheorem{lem}[thm]{Lemma}

\theoremstyle{definition}
\newtheorem{dfn}[thm]{Definition}
\newtheorem{ex}[thm]{Example}

\newtheorem{rem}[thm]{Remark}

\numberwithin{equation}{section}

\begin{document}

\title{Skew graded $(A_\infty)$ hypersurface singularities}
\author{Kenta Ueyama}
\address{
Department of Mathematics, 
Faculty of Education,
Hirosaki University, 
1 Bunkyocho, Hirosaki, Aomori 036-8560, Japan}
\email{k-ueyama@hirosaki-u.ac.jp} 
\thanks{The author was supported by JSPS KAKENHI Grant Numbers JP18K13381 and JP22K03222.}


\begin{abstract}
For a skew version of a graded $(A_\infty)$ hypersurface singularity $A$, we study the stable category of graded maximal Cohen-Macaulay modules over $A$.
As a consequence, we see that $A$ has countably infinite Cohen-Macaulay representation type and is not a noncommutative graded isolated singularity.
\end{abstract}

\maketitle
\setlength{\leftmargini}{20pt}  

\section{Introduction}
Representation theory of (graded) maximal Cohen-Macaulay modules is a very active and fruitful area of research \cite{I}. One of the fundamental subjects is to determine the Cohen-Macaulay representation type of (graded) rings \cite{LW, Y}.
Let $k$ be an algebraically closed field
of characteristic different from $2$, and
let $R=\bigoplus_{i\in\N}R_i$ be an $\N$-graded commutative Gorenstein ring with $R_0=k$.
Then $R$ is said to have \emph{finite Cohen-Macaulay representation type}
(respectively, \ \emph{countable Cohen-Macaulay representation type})
if it has only finitely (respectively, countably) many indecomposable graded
maximal Cohen-Macaulay modules up to isomorphism and shift of the grading.
The following two results are well-known.

\begin{thm}[\cite{EH}]\label{A1}
Let $R =k[x_1,\dots,x_n]/(x_1^2+\dots+x_{n}^2)$ be a graded $(A_1)$ hypersurface singularity with $\deg x_i=1$.
Then $R$ has finite Cohen-Macaulay representation type.
\end{thm}

\begin{thm}[{\cite[Theorem B]{BuGS}, \cite[Propositions 8 and 9]{AR}}] \label{Ainfty}
Let $R =k[x_1,\dots,x_n]/(x_1^2+\dots+x_{n-1}^2)$ be a graded $(A_\infty)$ hypersurface singularity with $n \geq 2$ and $\deg x_i=1$.
Then $R$ has countably infinite Cohen-Macaulay representation type.
\end{thm}

Graded $(A_1)$ and $(A_\infty)$ hypersurface singularities play an essential role in the study of higher-dimensional standard graded Gorenstein rings of countable Cohen-Macaulay representation type; see, for example, \cite[Section 5]{St}.

From here, we turn our attention to rings that are not necessarily commutative.
In \cite{HU}, Higashitani and the author computed the stable category of graded maximal Cohen-Macaulay modules over a skew version of a graded $(A_1)$ hypersurface singularity by combinatorial methods developed by Mori and the author \cite{MUk}.

\begin{dfn}\label{dfn.spm}
A \emph{$(\pm 1)$-skew polynomial algebra} in $n$ variables is defined to be an algebra
\[ k_{\e}[x_1,\dots,x_n] := k\langle x_1, \dots, x_n \rangle /(x_ix_j -\e_{ij} x_jx_i \mid 1\leq i,j \leq n), \]
where $\e =(\e_{ij}) \in M_n(k)$ is a symmetric matrix such that $\e_{ii}=1$ for all $1\leq i \leq n$ and $\e_{ij}=\e_{ji} \in \{1, -1\}$ for all $1\leq i<j \leq n$.
\end{dfn}

It is easy to see that both $x_1^2+\dots+x_{n}^2$ and $x_1^2+\dots+x_{n-1}^2$ are regular central elements of a $(\pm 1)$-skew polynomial algebra $k_{\e}[x_1,\dots,x_n]$.

\begin{dfn}\label{dfn.A1}
\begin{enumerate}
\item A \emph{skew graded $(A_1)$ hypersurface singularity} is defined to be a graded algebra
\[ k_{\e}[x_1,\dots,x_n]/ (x_1^2+\dots+x_{n}^2),\]
where $\deg x_i=1$ for all $1\leq i \leq n$.
\item A \emph{skew graded $(A_\infty)$ hypersurface singularity} is defined to be a graded algebra
\[ k_{\e}[x_1,\dots,x_n]/ (x_1^2+\dots+x_{n-1}^2),\]
where $n \geq 2$ and $\deg x_i=1$ for all $1\leq i \leq n$.
\end{enumerate}
\end{dfn}

For a skew graded $(A_1)$ or $(A_\infty)$ hypersurface singularity $A$,
we write $\uCM(A)$ for the stable category of graded maximal Cohen-Macaulay modules over $A$.

\begin{dfn}\label{dfn.graph}
For a $(\pm 1)$-skew polynomial algebra $k_{\e}[x_1,\dots,x_n]$,
the \emph{graph $G_{\e}$ associated to $k_{\e}[x_1,\dots,x_n]$} is defined to be the graph with vertex set 
$V(G_{\e})=\{1, \dots , n\}$ and edge set $E(G_{\e})=\{ij \mid \e_{ij}=\e_{ji}=1, i \neq j \}$. 
\end{dfn}

Let $\F_2$ denote the field with two elements $0$ and $1$. 
For a matrix $M$ with entries in $\F_2$, let $\rank_{\F_2}(M)$ (respectively, $\mnull_{\F_2}(M))$ denote the rank (respectively, the nullity) of $M$ over $\F_2$.

\begin{thm}[{\cite[Theorem 1.3]{HU}}] \label{thm.sA1s}
Let $A_\e= k_{\e}[x_1,\dots,x_n]/ (x_1^2+\dots+x_{n}^2)$ be a skew graded $(A_1)$ hypersurface singularity with $\deg x_i=1$, and let $G_{\e}$ the graph associated to $k_{\e}[x_1,\dots,x_n]$. Consider the matrix
\[ \De_\e = \begin{pmatrix}  
 & & &1 \\
 &M(G_\e) & &\vdots \\
 & & &1 \\
1 &\cdots &1 &0
\end{pmatrix}  \in M_{n+1}(\F_2),\]
where $M(G_\e)$ is the adjacency matrix of $G_\e$ (with entries in $\F_2$).
Then there exists an equivalence of triangulated categories
\[\uCM(A_\e) \cong \Db(\mod k^{2^r}),\]
where $r= \mnull_{\F_2} (\De_\e)$.
\end{thm}

It follows from Theorem \ref{thm.sA1s} that $A_\e$ has $2^{r}$ indecomposable non-projective graded maximal Cohen-Macaulay modules up to isomorphism and degree shift.
Therefore, we have the following result, which is a generalization of Theorem \ref{A1}.

\begin{cor}[{\cite[Theorem 1.3]{HU}}]  \label{cor.sA1}
Let $A_\e=k_{\e}[x_1,\dots,x_n]/ (x_1^2+\dots+x_{n}^2)$ be a skew graded $(A_1)$ hypersurface singularity. Then $A_\e$ has finite Cohen-Macaulay representation type.
\end{cor}

The purpose of this paper is to investigate the stable category of graded maximal Cohen-Macaulay modules over a skew graded $(A_\infty)$ hypersurface singularity in a manner analogous to Theorem \ref{thm.sA1s}. We prove the following theorem.

\begin{thm} \label{thm.sAinftys}
Let $A_\e= k_{\e}[x_1,\dots,x_n]/ (x_1^2+\dots+x_{n-1}^2)$ be a skew graded $(A_\infty)$ hypersurface singularity with $n \geq 2$ and $\deg x_i=1$. Let $G_{\e}$ the graph associated to $k_{\e}[x_1,\dots,x_n]$. Consider the matrix
\[ \De_\e = \begin{pmatrix}  
 & & &1 \\
 &M(G_\e) & &\vdots \\
 & & &1 \\
1 &\cdots &1 &0
\end{pmatrix}  \in M_{n+1}(\F_2),\]
where $M(G_\e)$ is the adjacency matrix of $G_\e$ (with entries in $\F_2$).
Let $\bm{v}_1, \dots, \bm{v}_{n+1} \in \F_2^{n+1}$ denote the columns of $\De_\e $.

\begin{enumerate}
\item If $\bm{v}_{n}$ can be expressed as a linear combination of the other columns $\bm{v}_1, \dots, \bm{v}_{n-1}, \bm{v}_{n+1}$, then there exists an equivalence of triangulated categories
\[\uCM(A_\e) \cong \Db(\mod \La^{2^{r-1}}), \]
where $\La$ is the finite-dimensional algebra given by the quiver with relations \vspace{2mm}
\begin{align} \label{quiv1}
\xymatrix@C=2pc@R=2pc{
1 \ar@(rd,ru)_{a} 
}
\qquad a^2=0, \vspace{2mm}
\end{align}
and $r= \mnull_{\F_2} (\De_\e)$.
\item If $\bm{v}_{n}$ cannot be expressed as a linear combination of the other columns $\bm{v}_1, \dots, \bm{v}_{n-1}, \bm{v}_{n+1}$, then there exists an equivalence of triangulated categories
\[\uCM(A_\e) \cong \Db(\mod \Ga^{2^r}), \]
where $\Ga$ is the finite-dimensional algebra given by the quiver with relations \vspace{2mm}
\begin{align} \label{quiv2}
\xymatrix@C=2pc{
1 \ar@<0.75ex>[r]^{a} &2 \ar@<0.75ex>[l]^{b} 
}
\qquad ab=0, ba=0, \vspace{2mm}
\end{align}
and $r= \mnull_{\F_2} (\De_\e)$.
\end{enumerate}
\end{thm}

Now consider the case $\ep_{ij}=1$ for all $1 \leq i<j \leq n$ and $n \geq 2$. In this case, $k_{\e}[x_1,\dots,x_n]=k[x_1,\dots,x_n]$ and $G_{\e}$ is the complete graph $K_n$, so
\[ \De_\e =\begin{pmatrix}  
 & & &1 \\
 &M(K_n) & &\vdots \\
 & & &1 \\
1 &\cdots &1 &0
\end{pmatrix}
= \underbrace{\begin{pmatrix}
0&1&\cdots &1\\
1&0&\ddots &\vdots\\
\vdots &\ddots &\ddots &1\\
1&\cdots &1 &0\\
\end{pmatrix}}
_{n+1}.\]
Let $\bm{v}_1, \dots, \bm{v}_{n+1} \in \F_2^{n+1}$ be the columns of $\De_\e$.
If $n$ is even, then we have $\bm{v}_{n}=\bm{v}_1+\dots+\bm{v}_{n-1}+\bm{v}_{n+1}$ and
$\mnull_{\F_2} (\De_\e)=1$.
If $n$ is odd, then one can check that $\bm{v}_{n}$ cannot be written as a linear combination of $\bm{v}_1, \dots, \bm{v}_{n-1}, \bm{v}_{n+1}$ and that $\mnull_{\F_2} (\De_\e)=0$.
Hence Theorem \ref{thm.sAinftys} contains the following result. 

\begin{cor} \label{cor.sAinftys}
Let $R= k[x_1,\dots,x_n]/ (x_1^2+\dots+x_{n-1}^2)$ be a graded $(A_\infty)$ hypersurface singularity with $n \geq 2$ and $\deg x_i=1$. 
\begin{enumerate}
\item If $n$ is even, then there exists an equivalence of triangulated categories
$\uCM(R) \cong \Db(\mod \La)$,
where $\La$ is the finite-dimensional algebra given by the quiver with relations \eqref{quiv1}.
\item If $n$ is odd, then there exists an equivalence of triangulated categories
$\uCM(R) \cong \Db(\mod \Ga)$,
where $\Ga$ is the finite-dimensional algebra given by the quiver with relations \eqref{quiv2}.
\end{enumerate}
\end{cor}

\begin{rem}
Corollary \ref{cor.sAinftys} also follows from 
graded Kn\"orrer's periodicity theorem and
the result of Buchweitz, Eisenbud, and Herzog \cite[Appendix]{BEH}.
In tilting theory, by Buchweitz, Iyama, and Yamaura's theorem \cite[Thorem 1.4]{BIY}, 
it is known that if $R=k[x,y]/(x^2)$ with $\deg x=\deg y=1$, then
there exists a triangle equivalence $F: \uCM_0(R) \xrightarrow{\sim} \Kb(\proj \La)$, where $\uCM_0(R)$ is the stable category of graded maximal Cohen-Macaulay $R$-modules $M$ that satisfy $M_{\p} \in \proj R_{\p}$ for all  $\p \in \Spec R$ with $\dim R_\p < \dim R$.
Furthermore, in \cite{IY}, it is shown that the equivalence $F: \uCM_0(R) \xrightarrow{\sim} \Kb(\proj \La)$ leads to an equivalence $\uCM(R) \cong  \Db(\mod \La)$.
Thus, Corollary \ref{cor.sAinftys} (1) can also be obtained from this and graded Kn\"orrer's periodicity theorem.
\end{rem}

As a consequence of Theorem \ref{thm.sAinftys}, we obtain the following corollary, which generalizes Theorem \ref{Ainfty}.

\begin{cor}\label{cor.sAinfty}
Let $A_\e=k_{\e}[x_1,\dots,x_n]/ (x_1^2+\dots+x_{n-1}^2)$ be a skew graded $(A_\infty)$ hypersurface singularity with $n \geq 2$. Then $A_\e$ has countably infinite Cohen-Macaulay representation type.
\end{cor}

In addition, Corollary \ref{cor.sAinfty} implies the following conclusion.
 \begin{cor}\label{cor.sAinfty2}
Let $A_\e=k_{\e}[x_1,\dots,x_n]/ (x_1^2+\dots+x_{n-1}^2)$ be a skew graded $(A_\infty)$ hypersurface singularity with $n \geq 2$. Then $A_\e$ is not a noncommutative graded isolated singularity; that is, the category $\qgr A_\e$ has infinite global dimension.
 \end{cor}

This paper is organized as follows.
In Section \ref{sec.pre}, some basic definitions and fundamental results are stated.
In Section \ref{sec.com}, the stable categories of graded maximal Cohen-Macaulay modules over skew graded $(A_\infty)$ hypersurface singularities are studied combinatorially.
Proofs of Theorem \ref{thm.sAinftys} and Corollaries \ref{cor.sAinfty} and \ref{cor.sAinfty2} are given in Section \ref{sec.p}.

\section{Preliminaries}\label{sec.pre}
Throughout this paper, $k$ is an algebraically closed field of characteristic different from $2$,
and all algebras are over $k$.

\subsection{Stable categories of graded Maximal Cohen-Macaulay modules}
For an $\N$-graded algebra $A$,
we write $\GrMod A$ for the category of graded right $A$-modules with $A$-module homomorphisms of degree zero, and $\grmod A$ for the full subcategory consisting of finitely generated graded modules.
For a graded module $M \in \GrMod A$ and an integer $s\in \Z$,
we define the \emph{shift} $M(s) \in \GrMod A$ to be the graded module with $i$-th degree component $M(s)_i = M_{s+i}$.
For $M, N\in \GrMod A$, we write $\Ext^i_{\GrMod A}(M, N)$ for the extension group in $\GrMod A$, and define $\Ext^i_A(M, N):=\bigoplus _{s \in \Z}\Ext^i_{\GrMod A}(M, N(s))$.

Let $A$ a noetherian $\N$-graded algebra.
Let $\qgr A = \grmod A/ \fdim A$ denote the quotient category of $\grmod A$ by the Serre subcategory $\fdim A$ of finite-dimensional modules. The category $\qgr A$ plays the role of (the category of coherent sheaves on) the noncommutative projective scheme associated to $A$; see \cite{AZ}.
A noetherian $\N$-graded algebra $A$ is called a \emph{noncommutative graded isolated singularity} if $\qgr A$ has finite global dimension; see \cite{U}.

Recall that an $\N$-graded algebra $A=\bigoplus_{i \in \N} A_i$ is said to be \emph{connected graded} if $A_0=k$. 
Let $A$ be a noetherian connected graded algebra. Then $A$ is called an \emph{AS-regular} (respectively, \emph{AS-Gorenstein}) algebra of dimension $d$ if
\begin{itemize}
\item{} $\gldim A = d <\infty$ (respectively, $\injdim_A A = \injdim_{A^{\op}} A= d <\infty$), and
\item{} $\Ext^i_A(k ,A) \cong \Ext^i_{A^{\op}}(k, A) \cong
\begin{cases}
k(\ell) & \text { if } i=d\\
0 & \text { if } i\neq d
\end{cases}$ for some $\ell \in \Z$.
\end{itemize}

Let $A$ be a noetherian AS-Gorenstein algebra.
We call $M \in \grmod A$ \emph{graded maximal Cohen-Macaulay} if $\Ext^i_A(M ,A) =0$ for all $i \neq 0$.
We write $\CM(A)$ for the full subcategory of $\grmod A$ consisting of graded maximal Cohen-Macaulay modules. Then $\CM(A)$ is a Frobenius category.
The \emph{stable category} of graded maximal Cohen-Macaulay modules, denoted by $\uCM(A)$, has the same objects as $\CM(A)$,
and the morphism space is given by
\[ \Hom_{\uCM(A)}(M, N) = \Hom_{\CM(A)}(M,N)/P(M,N), \]
where $P(M,N)$ consists of degree zero $A$-module homomorphisms factoring through a graded projective module. By \cite{Hap}, $\uCM(A)$ canonically has a structure of triangulated category.

\subsection{The algebra $C(A)$}
The main algebraic framework used in this paper is due to Smith and Van den Bergh \cite{SV}, which was originally developed by Buchweitz, Eisenbud, and Herzog \cite{BEH}.

Let $S$ be a $d$-dimensional noetherian AS-regular algebra with Hilbert series $(1-t)^{-d}$. Then $S$ is Koszul by \cite[Theorem 5.11]{S}.
Let $f \in S$ be a homogeneous regular central element of degree $2$, and let $A =S/(f)$.
Then $A$ is a $(d-1)$-dimensional noetherian AS-Gorenstein algebra.
Moreover, $A$ is Koszul by \cite[Lemma 5.1 (1)]{SV}, and there exists a homogeneous central regular element $w \in A^!_2$ such that $A^!/(w) \cong S^!$ by \cite[Lemma 5.1 (2)]{SV}.
We can define the algebra
\[ C(A) := A^![w^{-1}]_0. \]
By \cite[Lemma 5.1 (3)]{SV}, we have $\dim_k C(A) = \dim_k (S^!)^{(2)} = 2^{d-1}$.

\begin{thm}[{\cite[Proposition 5.2]{SV}}] \label{thm.SV}
With notation as above,  we have an equivalence $\uCM(A) \cong \Db(\mod C(A))$, where $\Db(\mod C(A))$ denotes the bounded derived category of finite-dimensional modules over $C(A)$.
\end{thm}

\begin{thm}[{\cite[Theorem 5.5]{MUk}; see also \cite[Proposition 5.2]{SV}, \cite[Theorem 6.3]{HY}}] \label{thm.ft}
With notation as above, the following are equivalent.
\begin{enumerate}
\item $A$ has finite Cohen-Macaulay representation type.
\item $A$ is a noncommutative graded isolated singularity.
\item $C(A)$ is semisimple.
\end{enumerate}
\end{thm}

\subsection{Graphs}
A \emph{graph} $G$ consists of a set of vertices $V(G)$ and a set of edges $E(G)$ between two vertices. 
In this paper, we always assume that $V(G)$ is a finite set and $G$ has neither loops nor multiple edges. 
An edge between two vertices $v, w\in V(G)$ is written by $vw \in E(G)$. 
For a vertex $v \in V(G)$, let $N_G(v)=\{u \in V(G) \mid uv \in E(G)\}$.
A graph $G'$ is called the \emph{induced subgraph} of $G$ induced by $V' \subset V(G)$ if $vw \in E(G')$ whenever $v,w \in V'$ and $vw \in E(G)$. 
For a subset $W \subset V(G)$, we denote by $G \setminus W$ the induced subgraph of $G$ induced by $V(G) \setminus W$. 

\begin{dfn}
\begin{enumerate}
\item We say that $v$ is a \emph{isolated vertex} of a graph $G$ if
$v$ is a vertex of $G$ such that $N_G(v)=\emptyset$.
\item We say that $vw$ is an \emph{isolated edge} of a graph $G$ if
$vw$ is an edge of $G$ such that $N_G(v)=\{w\}$ and $N_G(w)=\{v\}$. 
\end{enumerate}
\end{dfn} 

Now let us focus on the notions of switching and relative switching of graphs.

\begin{dfn}[{\cite[Section 11.5]{GR}}]
Let $G$ be a graph and $v\in V(G)$. 
The \emph{switching} $\mu_v(G)$ of $G$ at $v$ is defined to be the graph $\mu_v(G)$ with $V(\mu_v(G))=V(G)$ and 
\[E(\mu_v(G))=\{ vw \mid w \in V(G) \setminus N_G(v) \} \cup E(G \setminus \{v\}).\]
For $v, w \in V(G)$, we define $\mu_{w}\mu_{v}(G):= \mu_{w}(\mu_{v}(G))$.
\end{dfn}

\begin{rem}
\begin{enumerate}
\item The same notion is called mutation in \cite{MUk} and \cite{HU}.
\item For $v \in V(G)$, we have $\mu_{v}\mu_{v}(G)=G$. For $v, w \in V(G)$, we have $\mu_{w}\mu_{v}(G)=\mu_{v}\mu_{w}(G)$.
\end{enumerate}
\end{rem}

\begin{ex}
If
$G= \xy /r2pc/: 
{\xypolygon5{~={90}~*{\xypolynode}~>{}}},
"1";"3"**@{-},
"1";"5"**@{-},
"2";"4"**@{-},
"3";"4"**@{-},
"4";"5"**@{-},
\endxy$,
then
$\mu_1(G)= \xy /r2pc/: 
{\xypolygon5{~={90}~*{\xypolynode}~>{}}},
"1";"2"**@{-},
"1";"4"**@{-},
"2";"4"**@{-},
"3";"4"**@{-},
"4";"5"**@{-},
\endxy$.
\end{ex}

\begin{dfn}[{\cite[Definition 6.6]{MUk}}]
Let $v,w \in V(G)$ be distinct vertices.
Then the \emph{relative switching} $\mu_{v \leftarrow w}(G)$ of $G$ at $v$ with respect to $w$ is defined to be the graph $\mu_{v \leftarrow w}(G)$ with $V(\mu_{v \leftarrow w}(G))=V(G)$ and
\begin{align*}
E(\mu_{v \leftarrow w}(G))
=\{vu \mid u \in N_G(w) \setminus N_G(v)\} \cup \{vu \mid u \in N_G(v) \setminus N_G(w)\} \cup E(G \setminus \{v\}).
\end{align*}
For $v, v' w, w' \in V(G)$ with $v \neq w$ and $v' \neq w'$, we define $\mu_{v' \leftarrow w'}\mu_{v \leftarrow w}(G):= \mu_{v' \leftarrow w'}(\mu_{v \leftarrow w}(G))$.
\end{dfn}

\begin{rem}
\begin{enumerate}
\item In the original paper \cite{MUk}, relative switching is called relative mutation.
\item For distinct $v, w \in V(G)$, we have $\mu_{v \leftarrow w}\mu_{v \leftarrow w}(G) =G$.
\end{enumerate}
\end{rem}

\begin{ex}
If
$G= \xy /r2pc/: 
{\xypolygon6{~={90}~*{\xypolynode}~>{}}},
"1";"4"**@{-},
"1";"5"**@{-},
"2";"3"**@{-},
"2";"4"**@{-},
"3";"5"**@{-},
"4";"5"**@{-},
\endxy$,
then
$\mu_{1\leftarrow 2}(G)= \xy /r2pc/: 
{\xypolygon6{~={90}~*{\xypolynode}~>{}}},
"1";"3"**@{-},
"1";"5"**@{-},
"2";"3"**@{-},
"2";"4"**@{-},
"3";"5"**@{-},
"4";"5"**@{-},
\endxy$.
\end{ex}

For a graph $G$, let $M(G)$ denote the adjacency matrix of $G$ with entries in $\F_2$.  Define
\[
\De(G) =\begin{pmatrix} 
 & & &1 \\
 &M(G) & &\vdots \\
 & & &1 \\
1 &\cdots &1 &0
\end{pmatrix} \in M_{n+1}(\F_2).
\]

\begin{dfn}
For a matrix $M \in M_{n+1}(\F_2)$, we say that \emph{$M$ satisfies the condition \textnormal{(L)}} if the $n$-th column $\bm{v}_{n} \in \F_2^{n+1}$ of $M$ can be expressed as a linear combination of the other columns $\bm{v}_1, \dots, \bm{v}_{n-1}, \bm{v}_{n+1} \in \F_2^{n+1}$ of $M$.
\end{dfn}

\begin{lem} \label{lem.msw}
Let $G$ be a graph with $V(G)=\{1,\dots,n\}$.
If $G'=\mu_{v}(G)$ for some $v \in  V(G)$, then 
\begin{enumerate}
\item $\mnull_{\F_2} (\De(G)) = \mnull_{\F_2} (\De(G'))$, and
\item $\De(G)$ satisfies the condition (L) if and only if so does $\De(G')$.
\end{enumerate}
\end{lem}

\begin{proof}
By definition of switching,  we have
\begin{align*}
(E+E_{v,n+1})\De(G)(E+E_{n+1,v})&=\begin{pmatrix} 
 & & &1 \\
 &M(\mu_{v}(G)) & &\vdots \\
 & & &1 \\
1 &\cdots &1 &0
\end{pmatrix}=\De(G'), 
\end{align*}
where $E$ is the identity matrix and $E_{i,j}$ is the matrix such that the $(i,j)$-entry is $1$ and the other entries are all $0$. This yields the assertions of the lemma.
\end{proof}

\begin{lem} \label{lem.mrsw}
Let $G$ be a graph with $V(G)=\{1,\dots,n\}$ with $n\geq 2$.
Assume that $1$ is an isolated vertex of $G$.
Let $v, w\in V(G_{\e})$ be distinct vertices with $v\neq 1, w \neq n$.
If $G'=\mu_{v \leftarrow w}(G)$, then 
\begin{enumerate}
\item $\mnull_{\F_2} (\De(G)) = \mnull_{\F_2} (\De(G'))$, and
\item $\De(G)$ satisfies the condition (L) if and only if so does $\De(G')$.
\end{enumerate}
\end{lem}

\begin{proof}
Since $1$ is an isolated vertex of $G$, the first row (respectively, the first column) of $\De(G)$ is $\left(\begin{smallmatrix} 0 &\cdots &0 &1 \end{smallmatrix}\right)$ (respectively, $\left(\begin{smallmatrix} 0 \\ \vdots \\0 \\1 \end{smallmatrix}\right)$).
By definition of relative switching, we have
\begin{align*}
(E+E_{v,w}+E_{v,1})\De(G)(E+E_{w,v}+E_{1,v})&=\begin{pmatrix} 
 & & &1 \\
 &M(\mu_{v \leftarrow w}(G)) & &\vdots \\
 & & &1 \\
1 &\cdots &1 &0
\end{pmatrix}
= \De(G'),
\end{align*}
where $E$ is the identity matrix and $E_{i,j}$ is the matrix such that the $(i,j)$-entry is $1$ and the other entries are all $0$. This yields the assertions of the lemma.
\end{proof}

Switching and relative switching will be used to compute the stable categories of graded maximal Cohen-Macaulay modules over skew graded $(A_\infty)$ hypersurface singularities in the following sections.

\section{Graphical Methods for Computing Stable Categories}\label{sec.com}

Throughout this section, let $A_\e= k_{\e}[x_1,\dots,x_n]/ (x_1^2+\dots+x_{n-1}^2)$ be a skew graded $(A_\infty)$ hypersurface singularity with $n \geq 2$ and $\deg x_i=1$, and let $G_{\e}$ be the graph associated to $k_{\e}[x_1,\dots,x_n]$. The purpose of this section is to study $\uCM(A_\e)$ using $G_{\e}$.
Note that $k_{\e}[x_1,\dots,x_n]$ is an $n$-dimensional noetherian Koszul AS-regular domain with Hilbert series $(1-t)^{-n}$, so $A_\e$ is an $(n-1)$-dimensional noetherian Koszul AS-Gorenstein algebra.

\begin{lem} \label{lem.c}
\begin{enumerate}
\item $A_\e^!$ is isomorphic to 
\[k\ang{x_1, \dots, x_n}/(\ep_{ij}x_ix_j+x_jx_i,\,x_1^2-x_\ell^2,\,x_n^2 \mid 1\leq i, j\leq n,\, i\neq j,\,1\leq \ell \leq n-1).\]
\item $w:=x_1^2 \in A_\e^!$ is a central regular element such that $A_{\e}^!/(w) \cong k_{\e}[x_1,\dots,x_n]^!$.
\item $C(A_\e) = A_\e^![w^{-1}]_0$ is isomorphic to 
\[k\ang{t_2, \dots, t_{n}}/(t_it_j+\ep_{1i}\ep_{ij}\ep_{j1}t_jt_i,\,t_\ell^2-1,\,t_n^2 \mid 2\leq i, j\leq n,\,i\neq j,\,2\leq \ell \leq n-1).\]
\end{enumerate}
\end{lem}

\begin{proof}
(1) and (2) follow from direct calculation.

(3) Write $t_i := x_1x_{i}w^{-1} \in C(A_\e)$ for $2\leq i\leq n$.
Then it is easy to see that $\{t_2, \dots, t_n\}$ is a set of generators of $C(A_\e)$. Since we have
\begin{align*}
t_it_j &=  (x_1x_iw^{-1})(x_1x_jw^{-1}) =
-\ep_{1i}x_1^2x_ix_jw^{-2}
= -\ep_{1i}x_ix_jw^{-1} = \ep_{1i}\ep_{ji}x_jx_iw^{-1} \\
&= -\ep_{1i}\ep_{ij}\ep_{j1}(-\ep_{1j}x_1^2x_jx_iw^{-2}) = -\ep_{1i}\ep_{ij}\ep_{j1}(x_1x_jw^{-1})(x_1x_iw^{-1}) =  -\ep_{1i}\ep_{ij}\ep_{j1}t_jt_i,
\end{align*}
for $2\leq i, j\leq n, i\neq j$, 
\begin{align*}
t_\ell^2 &= (x_1x_\ell w^{-1})(x_1x_\ell w^{-1}) = -\ep_{1\ell}x_\ell^2w^{-1} = -\ep_{1\ell}x_1^2w^{-1} = -\ep_{1\ell}
\end{align*}
for $2\leq \ell\leq n-1$, and 
\begin{align*}
t_n^2 &= (x_1x_n w^{-1})(x_1x_n w^{-1}) = -\ep_{1n}x_n^2w^{-1} = 0,
\end{align*}
there exists a surjection
$k\ang{t_2, \dots, t_{n}}/(t_it_j+\e_{1i}\e_{ij}\e_{j1}t_jt_i,\,t_\ell^2+\e_{1\ell},\,t_n^2) \to C(A_\e)$.
This is an isomorphism because the algebras have the same dimension.
Since $\ep_{1\ell} \neq 0$, the assignment $t_\ell \mapsto \sqrt{-\ep_{1\ell}}t_\ell$ for $2\leq \ell\leq n-1$ and $t_n \mapsto t_n$ induces the isomorphism
\begin{align*}
&k\ang{t_2, \dots, t_{n}}/(t_it_j+\e_{1i}\e_{ij}\e_{j1}t_jt_i,\,t_\ell^2+\e_{1\ell},\,t_n^2)\\
&\xrightarrow{\sim} k\ang{t_2, \dots, t_{n}}/(t_it_j+\e_{1i}\e_{ij}\e_{j1}t_jt_i,\,t_\ell^2-1,\,t_n^2).
\end{align*}
Hence we have $C(A_\e) \cong k\ang{t_2, \dots, t_{n}}/(t_it_j+\e_{1i}\e_{ij}\e_{j1}t_jt_i,\,t_\ell^2-1,\,t_n^2)$.
\end{proof}

\begin{ex} \label{ex.C(A)}
\begin{enumerate}
\item If $A_\e= k[x_1,x_2]/ (x_1^2)$ or $A_\e= k\ang{x_1,x_2}/ (x_1x_2+x_2x_1,\,x_1^2)$, then $C(A_\e) \cong k[t]/(t^2)$ by Lemma \ref{lem.c} (3), so $C(A_\e)$ is isomorphic to the algebra $\La$ given by the quiver with relations \eqref{quiv1}.
\item If $A_\e= k[x_1,x_2, x_3]/ (x_1^2+x_2^2)$, then $C(A_\e) \cong k\ang{s,t}/(st+ts,\,s^2-1,\,t^2)$ by Lemma \ref{lem.c} (3). Let
\begin{align*}
e_1 = \frac{1}{2}(1+s+t+st),\quad
e_2 = \frac{1}{2}(1-s-t-st),\quad
a= \frac{1}{2}(t+st),\quad
b = \frac{1}{2}(t-st).
\end{align*}
Then $\{e_1,e_2,a,b\}$ is a $k$-basis of $k\ang{s,t}/(st+ts,\,s^2-1,\,t^2)$.
Since $e_1,e_2$ are orthogonal idempotents with $e_1+e_2=1$, and
\[ e_1ae_2=a,\quad e_2be_1=b,\quad ab=0,\quad ba=0,\]
it follows that $C(A_\e) \cong k\ang{s,t}/(st+ts,\,s^2-1,\,t^2)\cong \Ga$, where $\Ga$ is the algebra given by the quiver with relations \eqref{quiv2}.
\item If $A_\e= k\ang{x_1,x_2, x_3}/ (x_1x_2+x_2x_1,\,x_1x_3+x_3x_1,\, x_2x_3+x_3x_2,\,x_1^2+x_2^2)$, then $C(A_\e) \cong k[s,t]/(s^2-1,\,t^2)$ by Lemma \ref{lem.c} (3). Since $k[s,t]/(s^2-1,\,t^2) \cong k[s]/(s^2-1)\otimes k[t]/(t^2) \cong k^2\otimes k[t]/(t^2) \cong k[t]/(t^2)\times k[t]/(t^2)$,  we have $C(A_\e)\cong \La \times \La$, where
$\La$ is the algebra given by the quiver with relations \eqref{quiv1}.
\end{enumerate}
\end{ex}

We first give two theorems concerning switching and relative switching, which are analogs of \cite[Lemmas 6.5 and 6.7]{MUk}.

\begin{thm}[Switching theorem] \label{thm.S1}
Let $A_\e, A_{\e'}$ be skew graded $(A_\infty)$ hypersurface singularities.
If $G_{\e'}=\mu_{v}(G_{\e})$ for some $v \in  V(G_{\e})$, then $C(A_\e) \cong C(A_{\e'})$ and $\uCM(A_\e) \cong \uCM(A_{\e'})$.
\end{thm}

\begin{proof}
Since $G_{\e'}=\mu_{v}(G_{\e})$, we see that $\e'_{ij}\e'_{jh}\e'_{hi}=\e_{ij}\e_{jh}\e_{hi}$ for every $1\leq i<j<h\leq n$, so $C(A_\e) \cong C(A_{\e'})$ by Lemma \ref{lem.c} (3), and $\uCM(A_\e) \cong \uCM(A_{\e'})$ by Theorem \ref{thm.SV}.
\end{proof}

\begin{thm}[Relative switching theorem] \label{thm.S2}
Let $A_\e, A_{\e'}$ be skew graded $(A_\infty)$ hypersurface singularities.
Assume that $1$ is an isolated vertex of $G_{\e}$.
Let $v, w\in V(G_{\e})$ be distinct vertices with $v\neq 1, w \neq n$.
If $G_{\e'}=\mu_{v \leftarrow w}(G_\e)$, then $C(A_{\e})\cong C(A_{\e'})$ and $\uCM(A_{\e})\cong \uCM(A_{\e'})$.
\end{thm}

\begin{proof}
Since $\e_{1i}=-1$ for all $1\leq i \leq n$,
\[C(A_{\e}) \cong k\ang{t_2, \dots, t_{n}}/(t_it_j+\ep_{ij}t_jt_i,\,t_\ell^2-1,\,t_n^2 \mid 2\leq i, j\leq n,\,i\neq j,\,2\leq \ell \leq n-1)\]
by Lemma \ref{lem.c} (3). 
Let $D$ be the algebra generated by
$s_2,\dots, s_n$ with defining relations
\begin{align*}
&s_is_j+\e_{ij}s_js_i \quad (2\leq i, j\leq n,\,i\neq j,\,i\neq v,\,j \neq v),\\
&s_vs_j-\e_{vj}\e_{wj}s_js_v \quad (2\leq j\leq n,\,j\neq v,\,j\neq w),\\
&s_vs_w+\e_{vw}s_ws_v,\\
&s_\ell^2-1 \quad (2\leq \ell \leq n-1,\,\ell\neq v), \\
&s_v^2+\e_{vw} \quad (\text{when}\ v\neq n),\\
&s_n^2.
\end{align*}
Define a map
$\phi :k\ang{s_2, \dots, s_{n}} \to  k\ang{t_2, \dots, t_{n}}/(t_it_j+\ep_{ij}t_jt_i,\,t_\ell^2-1,\,t_n^2)$ by 
$s_i \mapsto t_i \;\; (2\leq i \leq n, i \neq v)$ and $s_{v} \mapsto  t_{v}t_{w}$.
Then one can verify that $\phi$ sends all defining relations of $D$ to zero, so 
we get an induced map $\overline{\phi}: D \to  k\ang{t_2, \dots, t_{n}}/(t_it_j+\ep_{ij}t_jt_i,\,t_\ell^2-1,\,t_n^2)$. It is easily seen that $\overline{\phi}$ is an isomorphism.

Moreover, since $G_{\e'}=\mu_{v \leftarrow w}(G_\e)$, it follows that 
\begin{align*}
&\e'_{ij}=\e_{ij} \quad (1\leq i, j\leq n,\,i\neq j,\,i\neq v,\,j \neq v),\\
&\e'_{vj}=-\e_{vj}\e_{wj} \quad (1\leq j \leq n,\,j\neq v,\,j\neq w),\\
&\e'_{vw}=\e_{vw}.
\end{align*}
and $1$ is an isolated vertex of $G_{\e'}$, so we see that $D$ is isomorphic to $C(A_{\e'})$ by Lemma \ref{lem.c} (3). Therefore, we obtain $C(A_{\e})\cong C(A_{\e'})$. The last equivalence follows from Theorem \ref{thm.SV}.
\end{proof}

We then give two ways to reduce the number of variables in the computation of $\CM(A_\e)$, which are analogs of \cite[Lemmas 6.17 and 6.18]{MUk}.
We will see that the first one is coming from the noncommutative Kn\"orrer's periodicity theorem \cite[Theorem 3.9]{MUk} and the second one is coming from \cite[Theorem 4.12 (3)]{MUk}.

\begin{thm}[{Special case of noncommutative graded Kn\"orrer's periodicity theorem}] \label{thm.Kn}
Let $S$ be a noetherian AS-regular algebra and $f$ a homogeneous regular central element of positive even degree $2m$. Then there exists an equivalence
\[\uCM(S/(f)) \cong \uCM(S[y,z]/(f+y^2+z^2)),\]
where $\deg y=\deg z=m$.
\end{thm}

\begin{proof}
This is a special case of \cite[Theorem 3.9]{MUk}.
\end{proof}

\begin{thm} [Kn\"orrer reduction] \label{thm.R1} 
Suppose that $vw$ is an isolated edge of $G_{\e}$, where $v\neq n, w \neq n$, and  $V(G_{\e})\setminus\{v,w,n\} \neq \emptyset$.
If $G_{\e'}=G_{\e}\setminus \{v, w\}$, 
then 
$\uCM(A_{\e})\cong \uCM(A_{\e'})$. 
\end{thm} 

\begin{proof} 
Let $G_{\e''}=\mu_{w}\mu_{v}(G_{\e})$. Then $x_{v}, x_{w}$ are central elements in $k_{\e''}[x_1,\dots, x_n]$. Since $G_{\e'}=G_{\e}\setminus \{v, w\}=G_{\e''}\setminus \{v, w\}$, we have $\uCM(A_{\e})\cong \uCM(A_{\e''}) \cong \uCM(A_{\e'})$
by Theorems \ref{thm.S1} and \ref{thm.Kn}.
\end{proof} 

Let $S$ be a connected algebra and $\si$ a graded algebra automorphism of $S$. 
A \emph{graded Ore extension} of $S$ by $\si$ is a graded algebra $S[y; \si]$ which is defined as follows:
$\deg y \geq 1$, $S[y; \si]=S[y]$ as a graded left free $S$-module, and the multiplication of $S[y; \si]$ is given by $ay=y\si(a)$ for $a\in S$. 
Let $\xi$ denote the graded algebra automorphism of $S$ defined by $a\mapsto (-1)^{\deg a}a$.  
For example, $k[x_1, x_2][y; \xi]$is isomorphic to $k\ang{x_1, x_2, y}/(x_1x_2-x_2x_1, x_1y+yx_1, x_2y+yx_2)$.

Let $S$ be a $d$-dimensional noetherian AS-regular algebra with Hilbert series $(1-t)^{-d}$.
For a homogeneous regular central element $f \in S$ of degree $2$ and $A=S/(f)$, we define $S^{\dagger}:=S[y; \xi]$ and $A^{\dagger}:=S^{\dagger}/(f+y^2)$, where $\deg y=1$. 
Since $f+y^2\in S^{\dagger}_2$ is a regular central element, we further define $S^{\dagger\dagger}:=(S^{\dagger})^{\dagger}=(S[y; \xi])[z; \xi]$ and $A^{\dagger\dagger}:=(A^{\dagger})^{\dagger}=S^{\dagger\dagger}/(f+y^2+z^2)$, where $\deg z=1$.

\begin{thm}[{\cite[Theorem 4.12 (3)]{MUk}}] \label{thm.MUdag}
Let $S$ be a $d$-dimensional noetherian AS-regular algebra with Hilbert series $(1-t)^{-d}$,
$f \in S$ a homogeneous regular central element of degree $2$, and $A=S/(f)$. Then $C(A^{\dagger\dagger}) \cong C(A^{\dagger}) \times C(A^{\dagger})$.
\end{thm}

\begin{thm}[Two point reduction]\label{thm.R2}
Let $A_\e, A_{\e'}$ be skew graded $(A_\infty)$ hypersurface singularities.
Suppose that $v, w \in V (G_\e)$ are two distinct isolated vertices different from $n$.
If $G_{\e'}=G_{\e}\setminus\{v\}$, then $C(A_\e) \cong C(A_{\e'}) \times C(A_{\e'})$ and $\uCM(A_\e) \cong \uCM(A_{\e'}) \times \uCM(A_{\e'})$.
\end{thm}

\begin{proof}
If $|V(G_\e)|=3$, then this follows from Example \ref{ex.C(A)} (1) and (3).
Assume that $|V(G_\e)|\geq 4$. Let $A=A_\e/(x_v, x_w)$. Then $A$ is a skew graded $(A_\infty)$ hypersurface singularity such that $A_{\e'} \cong A^{\dagger}$ and $A_{\e} \cong A^{\dagger\dagger}$, so $C(A_\e) \cong C(A_{\e'}) \times C(A_{\e'})$ by Lemma \ref{thm.MUdag}. Furthermore, we obtain
\begin{align*}
\uCM(A_\e) &\cong \Db(\mod C(A_\e)) \cong \Db(\mod(C(A_{\e'}) \times C(A_{\e'})))\\
&\cong \Db(\mod C(A_{\e'})) \times \Db(\mod C(A_{\e'}))
\cong \uCM(A_{\e'}) \times \uCM(A_{\e'})
\end{align*}
by Theorem \ref{thm.SV}.
\end{proof}

Here, we prove a combinatorial lemma, which is a modification of \cite[Lemma 3.1]{HU}.
For non-negative integers $a$ and $b$, let $G(a,b)$ denote the graph with set of vertices $\{u_i,u_i' \mid i=1,\ldots,a\} \cup \{u_j'' \mid j = 1,\ldots,b\}$ and set of edges $\{u_iu_i' \mid i=1,\ldots,a\}$. Namely, $G(a,b)$ consists of $a$ isolated edges and $b$ isolated vertices. 

\begin{lem} \label{lem.rsw}
Let $G$ be a graph with $V(G)=\{1,\dots,n\}$ and $n\geq 2$.
Assume that $1$ is an isolated vertex of $G$.
Then there exists a sequence of relative switchings $\mu_{v_1 \leftarrow w_1}, \dots, \mu_{v_m \leftarrow w_m}$ such that 
\[\mu_{v_m \leftarrow w_m}\mu_{v_{m-1} \leftarrow w_{m-1}}\cdots \mu_{v_1 \leftarrow w_1}(G)\cong G(\alpha,\beta),\] 
where $v_i \in \{2,\dots,n\}, w_j \in \{2,\dots,n-1\}$ for all $1\leq i,j\leq m$, $2\alpha+\beta=n$, and $\be \geq 1$. 
\end{lem}

\begin{proof}
We prove the claim by induction on $n$. If $n=2$, then the claim is trivial because $G$ is already equal to $G(0,2)$. Suppose that $n\geq3$. Let $G'=G\setminus\{n\}$.
By the induction hypothesis, there exists a sequence of relative switchings $\mu_{v_1 \leftarrow w_1}, \dots, \mu_{v_h \leftarrow w_h}$ such that 
\[\mu_{v_h \leftarrow w_h}\mu_{v_{h-1} \leftarrow w_{h-1}}\cdots \mu_{v_1 \leftarrow w_1}(G')\cong G(\alpha',\beta'),\]
where $v_i \in \{2,\dots,n-1\}, w_j \in \{2,\dots,n-2\}$ for all $1\leq i,j\leq h$, $2\alpha'+\beta'=n-1$, and $\be' \geq 1$.
Let 
\[G_1= \mu_{v_h \leftarrow w_h}\mu_{v_{h-1} \leftarrow w_{h-1}}\cdots \mu_{v_1 \leftarrow w_1}(G).\]
Then $G_1\setminus\{n\}\cong G(\alpha',\beta')$ and $1$ is still an isolated vertex of $G_1$.
Thus $G_1$ has the form
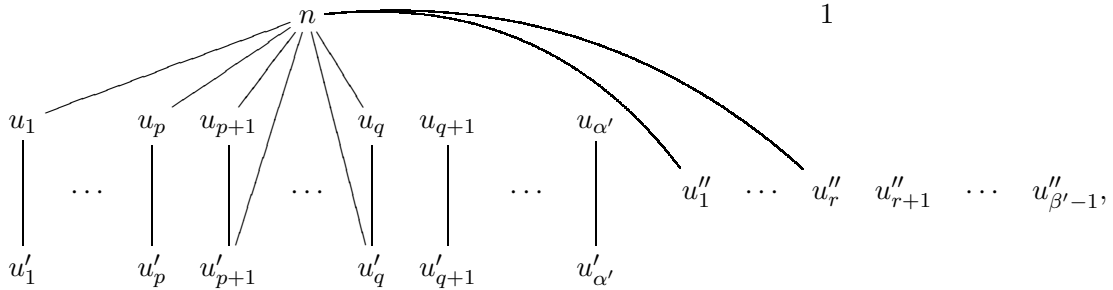
\begin{align*}&\xymatrix@R=0.8pc@C=0.5pc{
&&&&n \ar@{-}[ddllll] \ar@{-}[ddll] \ar@{-}[ddl] \ar@{-}[ddddl] \ar@{-}[ddr] \ar@{-}[ddddr] \ar@{-}@/^25pt/[dddrrrrrr] \ar@{-}@/^25pt/[dddrrrrrrrr]
&&&&&&&& 1\\
\\
u_1 \ar@{-}[dd]& &u_p \ar@{-}[dd] &u_{p+1} \ar@{-}[dd]&&u_q \ar@{-}[dd] &u_{q+1} \ar@{-}[dd] & &u_{\alpha'} \ar@{-}[dd]\\
 &\cdots &&&\cdots &&&\cdots &&&u''_1 &\cdots &u''_r &u''_{r+1} &\cdots &u''_{\beta'-1},\\
u_1' & &u_p'&u_{p+1}'&&u_q'&u_{q+1}' &&u'_{\alpha'}
}
\end{align*}
where $\{u_1, u_1',\dots, u_{\al'}, u_{\al'}'\} \cup \{u_1'',\dots, u_{\be'-1}''\}=\{2,\dots,n-1\}$ and
the set of edges of $G_1$ is 
\[
\{u_1u_1',\dots, u_{\al'}u_{\al'}'\} \cup \{nu_1,\dots,nu_p\} \cup \{nu_{p+1}, nu_{p+1}', \dots, nu_{q}, nu_{q}'\} \cup \{nu_1'',\dots,nu_r''\}
\]
for some $0 \leq p \leq q\leq \al'$ and $0 \leq r\leq \be'-1$.

Let 
\begin{align*}
&G_2=\mu_{n \leftarrow u_p'}\dots\mu_{n \leftarrow u_1'}(G_1),\\
&G_3=
\mu_{n \leftarrow u_{q}}\mu_{n \leftarrow u_{q}'}\dots\mu_{n \leftarrow u_{p+1}} \mu_{n \leftarrow u_{p+1}'}(G_2),\\
&G_4=
\mu_{u_r'' \leftarrow u_1''}\dots\mu_{u_2'' \leftarrow u_1''}(G_3).
\end{align*}
Note that $\mu_{n \leftarrow u_p'}\dots\mu_{n \leftarrow u_1'}$ removes edges $nu_1,\dots,nu_p$ from $G_1$, $\mu_{n \leftarrow u_{q}}\mu_{n \leftarrow u_{q}'}\dots\mu_{n \leftarrow u_{p+1}} \mu_{n \leftarrow u_{p+1}'}$ removes edges $nu_{p+1}, nu_{p+1}', \dots, nu_{q}, nu_{q}'$ from $G_2$, and $\mu_{u_r'' \leftarrow u_1''}\dots\mu_{u_2'' \leftarrow u_1''}$ removes edges $nu_2'',\dots,nu_r''$ from $G_3$.
Hence $G_4$ has the form
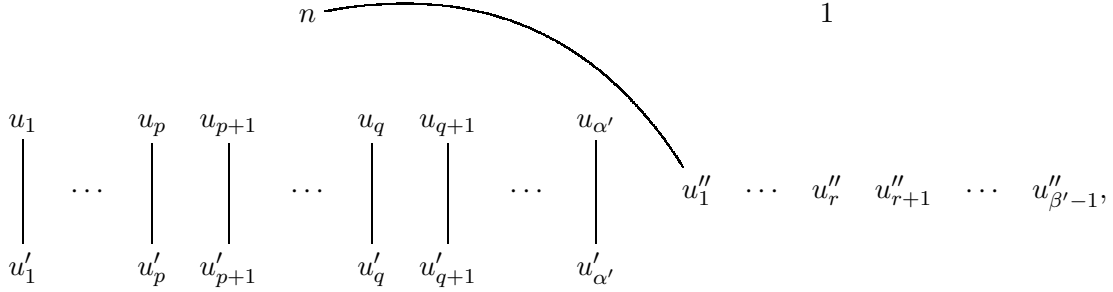
\begin{align*}
&\xymatrix@R=0.8pc@C=0.5pc{
&&&&n \ar@{-}@/^30pt/[dddrrrrrr] &&&&&&&& 1\\
\\
u_1 \ar@{-}[dd]& &u_p \ar@{-}[dd] &u_{p+1} \ar@{-}[dd]&&u_q \ar@{-}[dd] &u_{q+1} \ar@{-}[dd] & &u_{\alpha'} \ar@{-}[dd]\\
 &\cdots &&&\cdots &&&\cdots &&&u''_1 &\cdots &u''_r &u''_{r+1} &\cdots &u''_{\beta'-1},\\
u_1' & &u_p'&u_{p+1}'&&u_q'&u_{q+1}' &&u'_{\alpha'}\\
}
\end{align*}
where the set of edges of $G_4$ is 
$\{u_1u_1',\dots, u_{\al'}u_{\al'}'\} \cup \{nu_1''\}$ if $r\geq1$ and $\{u_1u_1',\dots, u_{\al'}u_{\al'}'\}$ if $r=0$.
That is, $G_4$ is isomorphic to $G(\alpha'+1,\beta'-1)$ if $r \geq 1$ and $G(\alpha',\beta'+1)$ if $r=0$. 
Therefore, we have proved the claim.
\end{proof}

\section{Proofs of Theorem \ref{thm.sAinftys} and Corollaries \ref{cor.sAinfty} and \ref{cor.sAinfty2}}\label{sec.p}

In this section, we provide proofs of Theorem \ref{thm.sAinftys} and Corollaries \ref{cor.sAinfty} and \ref{cor.sAinfty2}.

\begin{proof}[\textit{\textbf{Proof of Theorem \ref{thm.sAinftys}}}]
Let $N_{G_\e}(1)=\{u_1,\dots,u_h\}$ and define
\[ G_{\e'} := \mu_{u_h}\cdots \mu_{u_1}(G_\e).\]
Then $1$ is an isolated vertex of $G_{\e'}$. By Lemma \ref{lem.rsw},  there exists a sequence of relative switchings $\mu_{v_1 \leftarrow w_1}, \dots, \mu_{v_m \leftarrow w_m}$ such that 
\[
G_{\e''}:=\mu_{v_m \leftarrow w_m}\cdots \mu_{v_1 \leftarrow w_1}(G_{\e'})\cong G(\alpha,\beta),
\]
where $v_i \in \{2,\dots,n\}, w_j \in \{2,\dots,n-1\}$ for all $1\leq i,j\leq m$, $2\alpha+\beta=n$, and $\be \geq 1$.

By switching theorem (Theorem \ref{thm.S1}) and relative switching theorem (Theorem \ref{thm.S2}), we have
\begin{align}\label{eq.srs}
\uCM(A_\e)  \cong \uCM(A_{\e'})  \cong \uCM(A_{\e''}).
\end{align}
By Lemmas \ref{lem.msw} (1) and \ref{lem.mrsw} (1),
\begin{align}\label{eq.null}
r=\mnull_{\F_2}(\De_\e)= \mnull_{\F_2}(\De_{\e'})= \mnull_{\F_2}(\De_{\e''})=\mnull_{\F_2}(\De(G(\alpha,\beta)))= \beta-1
\end{align}
(note that $\be \geq1$). Moreover, by Lemmas \ref{lem.msw} (2) and \ref{lem.mrsw} (2),
\begin{align}\label{eq.L}
\De_\e\ \textnormal{satisfies (L)}\quad \Longleftrightarrow \quad
\De_{\e'}\ \textnormal{satisfies (L)} \quad \Longleftrightarrow \quad
\De_{\e''}\ \textnormal{satisfies (L)}.
\end{align}

(1) Since $\De_\e$ satisfies (L), so does $\De_{\e''}$ by \eqref{eq.L}, that is, the $n$-th column of $\De_{\e''}$ can be expressed as a linear combination of the other columns.
Since $G_{\e''} \cong G(\alpha,\beta)$ and 1 is an isolated vertex of $G_{\e''}$, we see that the $n$-th column of $\De_{\e''}$ is equal to $\left(\begin{smallmatrix} 0 \\ \vdots \\0 \\1 \end{smallmatrix}\right)$, so $n$ is an isolated vertex of $G_{\e''}$.

Using Kn\"orrer reduction (Theorem \ref{thm.R1}) $\al$ times and two point reduction (Theorem \ref{thm.R2}) $\be-2$ times, we have
\begin{align}\label{eq.red1}
\uCM(A_{\e''}) \cong \underbrace{\uCM(A_{\e'''}) \times \cdots \times \uCM(A_{\e'''})}_{2^{\be-2}},
\end{align}
where
\begin{align*}
G_{\e'''}:=\;\;\xymatrix@R=1.5pc@C=1.5pc{1 &n} \qquad(\textnormal{two isolated vertices}).
\end{align*}
Since $A_{\e'''}\cong k\ang{x_1,x_2}/ (x_1x_2+x_2x_1,\,x_1^2)$,
Example \ref{ex.C(A)} (1) and Theorem \ref{thm.SV} imply
\begin{align}\label{eq.La}
\uCM(A_{\e'''}) \cong \Db(\mod \La).
\end{align}
Hence we conclude 
\begin{align*}
\uCM(A_\e)  &\cong  \underbrace{\uCM(A_{\e'''}) \times \cdots \times \uCM(A_{\e'''})}_{2^{\be-2}} \qquad (\textnormal{by \eqref{eq.srs} and \eqref{eq.red1}} )\\
&\cong  \underbrace{\Db(\mod \La) \times \cdots \times \Db(\mod \La)}_{2^{\be-2}} \qquad (\textnormal{by \eqref{eq.La}} )\\
&\cong  \Db(\mod \La^{2^{\be-2}})\\
&\cong \Db(\mod \La^{2^{r-1}}) \qquad (\textnormal{by \eqref{eq.null}}),
\end{align*}
as desired.

(2) Since $\De_\e$ does not satisfy (L), it follows from \eqref{eq.L} that $\De_{\e''}$ does not satisfy (L), that is, the $n$-th column of $\De_{\e''}$ cannot be expressed as a linear combination of the other columns.
Since $G_{\e''} \cong G(\alpha,\beta)$ and 1 is an isolated vertex of $G_{\e''}$, we see that the $n$-th column of $\De_{\e''}$ is not equal to $\left(\begin{smallmatrix} 0 \\ \vdots \\0 \\1 \end{smallmatrix}\right)$, so $n$ is not an isolated vertex of $G_{\e''}$. In other words, $n$ constitutes an isolated edge of $G_{\e''}$.

Using Kn\"orrer reduction (Theorem \ref{thm.R1}) $\al-1$ times and two point reduction (Theorem \ref{thm.R2}) $\be-1$ times, we have
\begin{align}\label{eq.red2}
\uCM(A_{\e''}) \cong \underbrace{\uCM(A_{\e'''}) \times \cdots \times \uCM(A_{\e'''})}_{2^{\be-1}}
\end{align}
where
\[G_{\e'''}:=\;\;\xymatrix@R=1.5pc@C=1.5pc{1 &i \ar@{-}[r] &n} \qquad(\textnormal{one isoleted vertex and one isolated edge})\]
for some $i$. Let $G_{\e''''}= \mu_1(G_{\e'''})$. Then
\begin{align}\label{eq.Ga1}
\uCM(A_{\e'''}) \cong \uCM(A_{\e''''})
\end{align}
by switching theorem (Theorem \ref{thm.S1}). 
Since $A_{\e''''}\cong k[x_1,x_2,x_3]/ (x_1^2+x_2^2)$,
Example \ref{ex.C(A)} (2) and Theorem \ref{thm.SV} imply
\begin{align}\label{eq.Ga2}
\uCM(A_{\e''''}) \cong \Db(\mod \Ga).
\end{align}
Hence we conclude 
\begin{align*}
\uCM(A_\e)  &\cong \underbrace{\uCM(A_{\e'''}) \times \cdots \times \uCM(A_{\e'''})}_{2^{\be-1}} \qquad (\textnormal{by \eqref{eq.srs} and \eqref{eq.red2}} )\\
&\cong  \underbrace{\Db(\mod \Ga) \times \cdots \times \Db(\mod \Ga)}_{2^{\be-1}} \qquad (\textnormal{by \eqref{eq.Ga1} and \eqref{eq.Ga2}} )\\
&\cong  \Db(\mod \Ga^{2^{\be-1}})\\
&\cong \Db(\mod \Ga^{2^{r}}) \qquad (\textnormal{by \eqref{eq.null}}),
\end{align*}
as desired.
\end{proof}

\begin{proof}[\textit{\textbf{Proof of Corollary \ref{cor.sAinfty}}}]
By \cite[Theorem A]{BoGS}, we see that $\La$ and $\Ga$ are derived-discrete algebras.
Furthermore, it follows from \cite[Theorem B]{BoGS} that $\Db(\mod \La)$ and $\Db(\mod \Ga)$ have countably many indecomposable objects. Thus $\Db(\mod \La^{2^{r-1}})$ and $\Db(\mod \Ga^{2^{r}})$ also have countably many indecomposable objects.
By Theorem \ref{thm.sAinftys}, we have $\uCM(A_\e) \cong \Db(\mod \La^{2^{r-1}})$ or
$\uCM(A_\e) \cong \Db(\mod \Ga^{2^{r}})$, so in either case, $\uCM(A_\e)$ has countably many indecomposable objects. This says that $A_\e$ has countably many indecomposable non-projective graded maximal Cohen-Macaulay modules up to isomorphism, so $A_\e$ is
countable Cohen-Macaulay representation type. But $A_\e$ is not finite Cohen-Macaulay representation type. Indeed, suppose that $A_\e$ is finite Cohen-Macaulay representation type. Then $C(A_\e)$ is semisimple by Theorem \ref{thm.ft}, so $\uCM(A_\e) \cong \Db(\mod C(A_\e))$ is a semisimple triangulated category in the sense of \cite[Section 4.2]{SV}. However, since $\La^{2^{r-1}}$ and $\Ga^{2^{r}}$ are not semisimple, $\Db(\mod \La^{2^{r-1}})$ and $\Db(\mod \Ga^{2^{r}})$ are not semisimple triangulated categories. This is a contradiction.
\end{proof}

\begin{proof}[\textit{\textbf{Proof of Corollary \ref{cor.sAinfty2}}}]
By Corollary \ref{cor.sAinfty}, $A_\e$ is not finite Cohen-Macaulay representation type, so $A_\e$ is not a noncommutative graded isolated singularity by Theorem \ref{thm.ft}.
\end{proof}

In closing, we give an example.
\begin{ex}
Let us consider the case $n=4$.
Let $A_\e= k_{\e}[x_1,x_2,x_3,x_4]/ (x_1^2+x_2^2+x_3^2)$ be the skew graded $(A_\infty)$ hypersurface singularity given by
\[
\e= (\e_{ij})=
\begin{pmatrix}
1&1&-1&1\\
1&1&1&1\\
-1&1&1&-1\\
1&1&-1&1
\end{pmatrix}.
\]
Then 
\[G_{\e}= \xy /r2pc/: 
{\xypolygon4{~={90}~*{\xypolynode}~>{}}},
"1";"2"**@{-},
"1";"4"**@{-},
"2";"3"**@{-},
"2";"4"**@{-},
\endxy \qquad
\textnormal{and}
\qquad
\De_{\e}=
\begin{pmatrix}
0&1&0&1&1\\
1&0&1&1&1\\
0&1&0&0&1\\
1&1&0&0&1\\
1&1&1&1&0
\end{pmatrix}.
\]
One can check that the $4$-th column $\bm{v}_4=\left(\begin{smallmatrix}1\\1\\0\\0\\1 \end{smallmatrix}\right) \in \F_2^5$ of $\De_{\e}$ cannot be written as a linear combination of the other columns
$\bm{v}_1=\left(\begin{smallmatrix}0\\1\\0\\1\\1 \end{smallmatrix}\right), \bm{v}_2=\left(\begin{smallmatrix}1\\0\\1\\1\\1 \end{smallmatrix}\right),
\bm{v}_3=\left(\begin{smallmatrix}0\\1\\0\\0\\1 \end{smallmatrix}\right),
\bm{v}_5=\left(\begin{smallmatrix}1\\1\\1\\1\\0 \end{smallmatrix}\right)
\in \F_2^5$ of $\De_{\e}$, and $\mnull_{\F_2}(\De_{\e})=1$. Hence we have $\uCM(A_\e) \cong \Db(\mod \Ga^2)$ by Theorem \ref{thm.sAinftys}; compare with Corollary \ref{cor.sAinftys} (1) in the commutative case.
\end{ex}

\section*{Acknowledgments}
The author thanks Ryo Takahashi for valuable information on commutative rings of countable Cohen-Macaulay representation type.
The author also thanks Osamu Iyama for informing him of the results of the paper \cite{IY} in preparation.


\begin{thebibliography}{99}
\bibitem{AZ}
M. Artin and J. J. Zhang,
\textit{Noncommutative projective schemes},
Adv. Math. \textbf{109} (1994), no. 2, 228--287.

\bibitem{AR}
M. Auslander and I. Reiten,
\textit{Cohen-Macaulay modules for graded Cohen-Macaulay rings and their completions},
Commutative algebra (Berkeley, CA, 1987), Math. Sci. Res. Inst. Publ., vol. 15, Springer, New York, 1989, pp. 21--31.

\bibitem{BoGS}
G. Bobi\'nski, C. Gei\ss, and A. Skowro\'nski,
\textit{Classification of discrete derived categories}, 
Cent. Eur. J. Math. \textbf{2} (2004), no. 1, 19--49.

\bibitem{BEH}
R.-O. Buchweitz, D. Eisenbud, and J. Herzog,
\textit{Cohen-Macaulay modules on quadrics},
Singularities, representation of algebras, and vector bundles (Lambrecht, 1985),
Lecture Notes in Math., vol. 1273, Springer, Berlin, 1987, pp. 58--116.

\bibitem{BuGS}
R.-O. Buchweitz, G.-M. Greuel, and F.-O. Schreyer,
\textit{Cohen-Macaulay modules on hypersurface singularities II},
Invent. Math. \textbf{88} (1987), no. 1, 165--182.

\bibitem{BIY}
R.-O. Buchweitz, O. Iyama, and K. Yamaura,
\textit{Tilting theory for Gorenstein rings in dimension one},
Forum Math. Sigma \textbf{8} (2020), e36.

\bibitem{EH}
D. Eisenbud and J. Herzog,
\textit{The classification of homogeneous Cohen-Macaulay rings of finite representation type},
Math. Ann. \textbf{280} (1988), no. 2, 347--352.

\bibitem{GR}
C. Godsil and G. Royle, 
\textit{Algebraic graph theory}, 
Graduate Texts in Mathematics, vol. 207, Springer-Verlag, New York, 2001.

\bibitem{Hap}
D. Happel,
\textit{Triangulated categories in the representation theory of finite-dimensional algebras},
London Mathematical Society Lecture Note Series, vol. 119, Cambridge University Press, Cambridge, 1988.

\bibitem{HY}
J.-W. He and Y. Ye, 
\textit{Clifford deformations of Koszul Frobenius algebras and noncommutative quadrics},
preprint, \texttt{arXiv:1905.04699v2}.

\bibitem{HU}
A. Higashitani and K. Ueyama,
\textit{Combinatorial study of stable categories of graded Cohen-Macaulay modules over skew quadric hypersurfaces},
Collect. Math. \textbf{73} (2022), no. 1, 43--54.

\bibitem{I}
O. Iyama,
\textit{Tilting Cohen-Macaulay representations},
Proceedings of the International Congress of Mathematicians---Rio de Janeiro 2018. Vol. II. Invited lectures, World Sci. Publ., Hackensack, NJ, 2018, pp. 125--162.

\bibitem{IY}
O. Iyama and K. Yamaura,
\textit{Tilting theory for large Cohen-Macaulay modules over Gorenstein rings in dimension one},
in preparation.

\bibitem{LW}
G. Leuschke and R. Wiegand, 
\textit{Cohen-Macaulay representations},
Mathematical Surveys and Monographs, vol. 181, American Mathematical Society, Providence, RI, 2012.

\bibitem{MUk} 
I. Mori and K. Ueyama, 
\textit{Noncommutative Kn\"orrer's periodicity theorem and noncommutative quadric hypersurfaces},
to appear in Algebra Number Theory, \texttt{arXiv:1905.12266v3}.

\bibitem{S}
S. P. Smith,
\textit{Some finite-dimensional algebras related to elliptic curves},
Representation theory of algebras and related topics (Mexico City, 1994),
CMS Conf. Proc., vol. 19, American Mathematical Society, Providence, RI, 1996, pp. 315--348.

\bibitem{SV}
S. P. Smith and M. Van den Bergh,
\textit{Noncommutative quadric surfaces},
J. Noncommut. Geom. \textbf{7} (2013), no. 3, 817--856.

\bibitem{St}
B. Stone, 
\textit{Non-Gorenstein isolated singularities of graded countable Cohen-Macaulay type},
Connections Between Algebra, Combinatorics, and Geometry,
Springer Proc. Math. Stat., vol. 76, Springer, New York, 2014, pp. 299--317.

\bibitem{U} 
K. Ueyama,
\textit{Graded maximal Cohen-Macaulay modules over noncommutative graded Gorenstein isolated singularities}, 
J. Algebra \textbf{383} (2013), 85--103. 

\bibitem{Y}
Y. Yoshino,
\textit{Cohen-Macaulay Modules over Cohen-Macaulay Rings},
London Mathematical Society Lecture Note Series, vol. 146. Cambridge University Press, Cambridge, 1990.
\end{thebibliography}
\end{document}